\documentclass[10pt, reqno]{amsart}
\numberwithin{equation}{section}
\newtheorem{theorem}{Theorem}[section]
\newtheorem{corollary}[theorem]{Corollary}
\newtheorem{lemma}[theorem]{Lemma}

\usepackage{amssymb,latexsym}
\usepackage{eucal}
\usepackage{tikz}
\usepackage{comment}

\numberwithin{equation}{section}

\DeclareMathOperator{\dg}{deg}

\newcommand{\al}{\alpha}
\newcommand{\dl}{\zeta}

\newcommand{\e}{\eta}

\newcommand{\ph}{\varphi}

\DeclareMathOperator{\core}{core}

\newcommand{\barr}{\bar {\mathrm{r}}}
\newcommand{\rr}{\mathrm{r}}

\begin{document}

\title[On rank of the join of two subgroups in a free group]
 {On rank of the join of two subgroups in a free group}
\author{S. V. Ivanov}
  \address{Department of Mathematics\\
  University of Illinois \\
  Urbana\\   IL 61801\\ U.S.A.} \email{ivanov@illinois.edu}
\thanks{Supported in part by  the NSF under  grant  DMS 09-01782.}

\begin{abstract}
Let $H, K$ be two finitely generated subgroups  of a free group, let $\langle H, K \rangle$ denote  the
subgroup generated by  $H, K$, called the join of $H, K$,  and let neither of $H$, $K$  have finite index in $\langle H, K \rangle$.
We prove the existence of an epimorphism $\dl : \langle H, K \rangle \to F_2$,
where $F_2$ is a free group of rank 2, such that the restriction of $\dl$ on both $H$ and $K$ is injective
and the restriction    $\dl_0 : H \cap  K  \to  \dl (H)  \cap \dl (K) $   of $\dl$ on   $H \cap  K $
to  $\dl (H)  \cap \dl (K)$ is surjective. This is obtained as a corollary of an analogous result on rank of  the generalized join of two  finitely generated subgroups in a free group.
\end{abstract}

\maketitle


\section{Introduction}

It is a recurring  theme in group theory to embed a countable group into a 2-generated group
often with some additional properties, see \cite{CI},   \cite{Guba}, \cite{HNN}, \cite{LS}, \cite{Ol}.
In this article we will look at two finitely generated  subgroups of a free group  from
a similar perspective.

Let $H, K$ be two finitely generated subgroups  of a free group
$F$. Let  $S(H,K)$ denote a set of representatives of those double cosets $H g K$,
$g \in F$, for which the intersection $H \cap g K g^{-1}$ is nontrivial.

According to Walter Neumann \cite{WN}, the set $S(H,K)$ is finite and a formal disjoint union
$\bigvee_{s \in S(H,K)} H \cap s K s^{-1}$   of subgroups
$H \cap s K s^{-1}$, $s \in S(H,K)$, could be considered as a generalized intersection of $H$ and $K$.
Let $\rr(F)$ denote the rank of a free group $F$ and let $\barr(F) := \max (\rr(F)-1,0)$
denote the reduced rank of $F$.

\begin{theorem}\label{Th1} Suppose that $H, K$ are two finitely generated subgroups  of a free group
$F$, $L = \langle H, K, S(H,K) \rangle$ denotes  the
subgroup generated by  $H, K, S(H,K)$, and neither $H$ nor $K$  has finite index in $L$.
Then there is an epimorphism $\e : L \to F_2$,
where $F_2$ is a free group of rank 2, such that the restriction of $\e$ on $H$ and $K$ is injective, a set
$S(\e(H), \e(K) )$ for subgroups $\e(H), \e(K)$ of $F_2$ can be taken to be $\e(S(H,K))$,  and,
 for every $s \in S(H,K)$, the restriction
 $$
 \e_s : H \cap s K s^{-1} \to  \e(H)  \cap \e(s) \e(K) \e(s)^{-1}
 $$
of $\e$ on   $H \cap s K s^{-1}$ to $\e(H)  \cap \e(s) \e(K) \e(s)^{-1}$ is surjective.
\end{theorem}

Informally, we can say that, when given a {\em generalized join}
$$
L = \langle H, K, S(H,K) \rangle
$$
of two finitely generated subgroups $H$, $K$ of a free group $F$ such that neither of $H, K$ has finite index in $L$, it is always possible to assume
that  $\rr(L) =2$,   i.e., the subgroup $\langle H, K, S(H,K) \rangle$ is 2-generated.
Since a free group of an arbitrary finite rank is isomorphic to a subgroup of a
free group of rank 2, it is easy to obtain the equality   $\rr (F) = 2$ and our goal here
will be to attain the equality $\rr (L) = 2$.

We mention an easy corollary of  Theorem~\ref{Th1} for a ``pure" intersection case.

\begin{corollary}\label{Cor} Suppose that $H, K$ are two finitely generated subgroups  of a free group
$F$, $\langle H, K \rangle$ denotes  the
subgroup generated by  $H, K$, and neither $H$ nor $K$  has finite index in their {\em join} $\langle H, K \rangle$.
Then there is an epimorphism $\dl : \langle H, K \rangle \to F_2$,
where $F_2$ is a free group of rank 2, such that the restriction of $\dl$ on $H$ and $K$ is injective  and the restriction    $\dl_0 : H \cap  K  \to  \dl (H)  \cap \dl (K) $   of $\dl$ on   $H \cap  K $ to
$\dl (H)  \cap \dl (K) $   is surjective.
\end{corollary}

It is worthwhile to mention that Theorem~\ref{Th1} is false in case when one of subgroups $H, K$ has finite index in $L = \langle H, K, S(H,K) \rangle$.  Indeed, if $H$ has finite index $j$ in $L$ and $\rr(L) = n > 2$  then, according to the Schreier's formula, we have $\rr(H) = (n-1) j +1$, see \cite{LS}.
Since $\e$ is an epimorphism, the subgroup $\e(H)$ has finite index $j' \le j$ in $F_2 = \e(L)$  and it follows from the  Schreier's formula that $\rr(\e(H)) = j' +1$. Hence, the equality $\rr(H) =\rr(\e(H))$ is impossible as $(n-1) j > j'$ and
the restriction of $\e$ may not be injective on $H$.

We also remark that Theorem~\ref{Th1} is motivated by the author's article \cite{Iv10} in which  certain modifications
of Stallings graphs of subgroups $H, K$, that do not change the ranks
$\rr(H), \rr(K)$,   $\rr(H \cap s K s^{-1})$ for every $s \in S(H,K)$, are used to achieve some desired properties
of the Stallings graph of $L$ whose rank $\rr(L)$, however, might increase under carried out modifications.
In this article,  we make different modifications, in somewhat opposite direction, that decrease the rank $\rr(L)$ down to 2 while the ranks $\rr(H), \rr(K)$,  $\rr(H \cap s K s^{-1})$ for every $s \in S(H,K)$, and the cardinality $|S(H,K)|$  are kept fixed.

\section{Preliminaries}

Suppose that $Q$ is a graph. Let $VQ$ denote the set of vertices of $Q$ and let  $EQ$ denote the set
of oriented edges of $Q$. If $e \in EQ$ then $e^{-1}$
denotes the edge with the opposite to $e$ orientation, $e^{-1} \ne e$.

For $e \in EQ$, let $e_-$ and $ e_+$ denote the initial and terminal, respectively, vertices of $e$. A path
$p = e_1 \dots e_k$, where $e_i \in EQ$,
$(e_i)_+ =  (e_{i+1})_-$, $i =1, \dots, k-1$, is called {\em reduced } if, for every  $i =1, \dots, k-1$,
$e_i \ne e_{i+1}^{-1}$. The {\em length} of $p$ is $k$, denoted $|p| =k$.
The initial vertex of  $p$ is  $p_- = (e_1)_-$ and the terminal vertex of $p$ is  $p_+ = (e_k)_+$.
A path $p$ is {\em closed } if $p_- = p_+$.
If $p = e_1 \dots e_k$ is a closed path then a {\em cyclic permutation}
$\bar p$ of $p$ is a path of the form $e_{1+i}e_{2+i}  \dots e_{k+i}$, where $i =1,\dots,k$
and the indices are considered $\mod k$.

The subgraph of $Q$ that consists of edges of all closed paths $p$ of $Q$ such that  $|p| >0$
and any cyclic permutation of $p$ is reduced  is called the {\em core} of $X$, denoted $\core(X)$.

Let $F$ be a free group of finite rank $\rr(F) >1$. We consider $F$ as the fundamental group $\pi_1(U)$
where $U$ is a bouquet of $\rr(F)$ circles.

Following Stallings \cite{St}, see also \cite{WD}, \cite{KM}, with every (finitely generated) subgroup $H$ of
$F = \pi_1(U)$, we can associate a connected (resp. finite) graph $X = X(H)$
with a distinguished vertex $o \in VX$  and a locally injective map $\ph : X \to U$ of graphs so that
$H$ is isomorphic to $\pi_1(X, o )$.
Such a graph $X$ of $H$ is called a {\em Stallings graph} of $H$ and the map  $\ph$
is called a {\em canonical immersion}.

Consider two finitely generated subgroups $H, K$ of the free group
$F$. Pick a set $S(H,K)$ of representatives of those double cosets $H g K$,
$g \in F$, for which the intersection $H \cap g K g^{-1}$ is nontrivial.

Let $X, Y$ be finite Stallings graphs of the subgroups $H, K$, resp.,  and let $X  \times_U  Y$
denote the  pullback of  canonical immersions
\begin{gather}\label{ph1}
\ph_X : X \to U ,  \quad  \ph_Y : Y \to U  .
 \end{gather}

According to Walter Neumann \cite{WN},  the set $S(H,K)$ is finite and the
nontrivial intersections  $H \cap s K s^{-1}$, where $s \in S(H, K)$,
are in bijective correspondence with connected components $W_s$ of the core
$$
W := \core(X  \times_U  Y)  .
$$
Moreover,  for every $s \in S(H,K)$, we have
\begin{gather*}
\barr ( H \cap s K s^{-1} ) =  \tfrac 12 |  E W_s| - | V W_s | ,
\end{gather*}
where $\barr(F) = \max( \rr (F)-1, 0)$ is the reduced rank of a free group $F$ and
$|T|$ is the cardinality of a finite set $T$.
Recall that, according to our notation, the number of nonoriented edges of  $W_s$ is $\tfrac 12 |  E W_s|$.

For a finite graph $Q$,  denote
\begin{gather*}
\barr (Q) : = \tfrac 12 |  E Q| - | V Q | ,
\end{gather*}
hence, $\barr (Q)$ is the  negative Euler characteristic of $Q$.

In particular,  $\barr (W_s) = \barr ( H \cap s K s^{-1} )$ and
\begin{gather*}
\sum_{s \in S(H,K)} \barr (H \cap s K s^{-1} ) = \barr (W) = \tfrac 12 | E W | - |  V W |  .
\end{gather*}

Let $\al'_X$, $\al'_Y$  denote the projection maps $X \times_U  Y \to X$,
$X \times_U  Y \to Y$, resp. Restricting  $\al'_X, \al'_Y$ to
$W \subseteq  X  \times_U  Y $, we obtain immersions
$$
\al_X : W \to X , \quad   \al_Y : W \to Y .
$$

We also consider the  subgroup $L = \langle H, K, S(H,K) \rangle$ of $F$.
Let $Z$ denote a Stallings graph of  $\langle H, K, S(H,K) \rangle$
and let $\gamma : Z \to U$ denote a canonical immersion.

Let $\beta_X : X \to Z$, $\beta_Y : Y \to Z$ be graph maps that satisfy the equalities
$\ph_X = \gamma \beta_X$, $\ph_Y = \gamma \beta_Y$, see \eqref{ph1} and Fig.~1.
Clearly, $\beta_X$ and  $\beta_Y$ are immersions.

It follows from the definitions that, for every  $Q \in \{  U, W, X, Y,  Z \}$,  there is a canonical immersion
$\ph : Q \to U$, where $\ph = \ph_Q$ if $Q = X $ or $Q = Y$,
$\ph = \mbox{id}_U$ if $Q = U$,  $\ph =  \gamma  \beta_X    \al_X =   \gamma  \beta_Y    \al_Y$ if $Q = W$,
and  $\ph = \gamma$ if $Q = Z$, see Fig.~1.

\begin{center}
\begin{tikzpicture}[scale=.72]
\draw [-latex](.5,.5) -- (1.5, 1.5);
\draw [-latex](.5,-.5) -- (1.5, -1.5);
\draw [-latex](2.5,1.5) -- (3.5, 0.5);
\draw [-latex](2.5,-1.5) -- (3.5, -0.5);
\draw [-latex](4.5,0) -- (6.3, 0);
\draw [-latex](2.5,1.8) -- (6.3, 0.2);
\draw [-latex](2.5,-1.8) -- (6.3, -0.2);

\node at (-.2,0) {$W$};
\node at (2,2.05) {$X$};
\node at (2,-2.05) {$Y$};
\node at (.6,1.1) {$\alpha_X$};
\node at (.6,-1.1) {$\alpha_Y$};
\node at (3.5,1.1) {$\beta_X$};
\node at (3.5,-1.1) {$\beta_Y$};
\node at (4.5,1.3) {$\varphi_X$};
\node at (4.5,-1.3) {$\varphi_Y$};

\node at (5.2,0.3) {$\gamma$};
\node at (4,0) {$Z$};
\node at (6.75,0) {$U$};
\node at (2.4,-3.) {Fig. 1};
\end{tikzpicture}
\end{center}

It is not difficult to see that, up to a suitable conjugation of  the ambient free group  $F$, we may assume that
the graphs $X, Y$ coincide with their cores. Clearly, the same property holds for graphs
$U, W, Z$ as well.

Without loss of generality, we may also assume that
$Z =  U$ and  $\gamma = {\mbox{id}}_Z$.  For this reason, we will disregard $U$ and $\gamma$ in subsequent arguments.

\section{Five Lemmas}

Let $A = \{ a_1, a_1^{-1}, \ldots,  a_n, a_n^{-1} \}$ be an alphabet and let $F(A) = \langle a_1,  \dots,  a_n \rangle$ be a free group with free generators $a_1,  \dots,  a_n$.
It will be convenient to consider elements of the free group $F(A)$  as   words over the alphabet  $A = \{ a_1, a_1^{-1}, \dots,  a_n, a_n^{-1} \}$.
A letter-by-letter equality of words $u, w$ over $A$ is denoted  $u \equiv v$.
Suppose  $w \equiv c_1 \ldots c_\ell$ is a word over $A$, where $c_1, \ldots, c_\ell \in A$ are letters.
The length of $w$ is denoted $|w| = \ell$.
We say that a word $w \equiv c_1 \ldots c_\ell$ is {\em reduced}  if $|w| >0$ and $c_i \ne c_{i+1}^{-1}$ for every $i = 1, \ldots, \ell-1$.

A finite graph $B$ is called a {\em labeled $A$-graph}, or just an {\em  $A$-graph},
if $B$ is equipped with a function $\varphi : EB \to A$ so that,
for every $e \in EB$, we have $\varphi( e^{-1}) =  {\varphi( e)}^{-1}$.
An  $A$-graph $B$ is called   {\em irreducible} if, for every pair
$e_1, e_2 \in EB$, the equalities $\varphi(e_1) = \varphi(e_2)$ and $(e_1)_- = (e_2)_-$
imply that $e_1 = e_2$. Note that an irreducible $A$-graph need not be connected and may
contain vertices of degree $< 2$.

It is easy to see that $B$ is an irreducible $A$-graph with a labeling function
$\ph$ if and only if there is an immersion $\ph_0 : B \to A_0$, where $A_0$ is a bouquet of
$n$ oriented circles  $a_1, \dots, a_n$   so that the restriction of $\ph_0$ on  $EB$ is  $\ph$.

We now discuss some operations over irreducible $A$-graphs. Let $B$ be a finite irreducible $A$-graph.
A connected component $C$  of $B$ is called {\em $A$-complete}
if every vertex of $C$  has degree $|A| = 2n$. Equivalently,
the restriction of $\ph_0$ on $C$, $\ph_{0|_{C}} : C \to A_0  $, is a covering.
If a connected component $C$  of $B$ is not $A$-complete, we will say that $C$ is  {\em $A$-incomplete}.

\smallskip

Suppose $V_1 \subseteq VB$ is a subset of vertices of $B$, $w$ is a reduced word over $A$,
$w \equiv c_1 \dots c_\ell$, where $c_1, \dots, c_\ell  \in A$ are letters.  For every $v \in V_1$,
we consider a new graph which is a path $p(v)  = e_1(v) \dots e_\ell(v)$, consisting of edges
$e_1(v), \dots, e_\ell(v)$ labeled by  the letters $c_1, \dots, c_\ell$, resp., so $\ph(p(v)) \equiv w$.
For every $v \in V_1$, we attach the path  $p(v) $  to $B$ by identifying the vertices $p(v)_-$ and $v $.
This way we obtain a labeled $A$-graph $B'(V_1, w)$. We will then apply a folding process  to  $B'(V_1, w)$ that inductively identifies edges $e$ and $e'$ whenever $\ph(e) = \ph(e')$ and $e_- = e'_-$. As a result, we obtain an irreducible $A$-graph $B(V_1, w)$ that contains the original $A$-graph $B$, $B \subseteq B(V_1, w)$, and  has the following properties:  $\barr(B(V_1, w) ) = \barr(B )$ and, for
every $v \in V_1$, there is a unique path $p(v)$ in
$B(V_1, w)$ such that $p(v)_- = v$ and  $\ph(p(v)) \equiv w$.

We also observe that $B(V_1, w) =  B \cup F(V_1, w)$, where $F(V_1, w)$ is a forest,
i.e., a disjoint collection of trees, and $B \cap F(V_1, w) =  V_1^*$, where $V_1^*   \subseteq VB$.

\begin{lemma}\label{BF} Suppose that $B$ is a finite irreducible  $A$-graph
and no connected component of $B$ is $A$-complete. Then, for every subset $V_1 \subseteq VB$,
there is a reduced word $w = w(V_1)$ over $A$ and   there is an irreducible $A$-graph
$B(V_1, w)$  that  has the following properties.
The graph
$B(V_1, w)$ contains $B$ as a subgraph, $B(V_1, w) = B \cup F(V_1, w)$, where $F(V_1, w)$ is a forest, $B \cap F(V_1, w) =  V_1^*$,  $V_1^*   \subseteq VB$,
and $\barr(B(V_1, w) ) = \barr(B )$.

Furthermore, for every vertex $v \in V_1$, there is a unique path $p(v)$ in $B(V_1, w)$  such that $p(v)_- = v$,   $\ph(p(v)) \equiv w$, $p(v)_+$ has degree one,  and  $p(v)_+ \not\in B$.
\end{lemma}

\begin{proof} We prove this Lemma by induction on $| V_1 | \ge 1$.
To make the basis step, we let  $V_1 := \{ v \}$.
By the hypothesis of the Lemma, there is  a path $q $ in $B$ such that $\ph(q)$ is reduced,
$q_- = v$ and  $q_- = u$, where $u$ is a vertex of $B$ with $\dg (u)< 2n$.
Then there exists a letter $c \in A$ such that  there is no edge in $B$
with $e_- = u$ and  $\ph(e) = c$. Taking the word $w = \ph(q) c $,
we will obtain the desired result.

To make the induction step, assume that our claim holds for a set $V_1 = \{ v_1, \dots, v_k\}$ with $|V_1|=k$ and $V_2 = \{ v_1, \dots, v_k, v_{k+1} \}$, where $v_{k+1} \not\in V_1$. By the induction hypothesis, there exits a word $w_1 = c_1 \dots c_\ell$ and an irreducible $A$-graph $B(V_1, w_1)$ with the properties stated in Lemma. Let $m$ denote the maximum of distances between vertices in the same connected component of $B$ over all components of $B$. We replace $w_1$ by the reduced word $w_2 = w_1 c_\ell^{m+2}$ and construct the graph $B(V_1, w_2)$  for this new word $w_2$.
It is clear that $B(V_1, w_2)$ has all of the properties of $B(V_1, w_1)$. In addition, for every path $p(v_i)$ in $B(V_1, w_2)$, there is a factorization $p(v_i) = p_1(v_i)  p_2(v_i)$ so that
$\ph(p_2(v_i)) = c_\ell^{m+2}$, all the vertices of  $ p_2(v_i)$, except for $p_2(v_i)_+$, have degree 2 and $\dg  p_2(v_i)_+= 1$.

Now we take a new graph which is a path $p(v_{k+1})$ with  $\ph(p(v_{k+1})) = w_2$ and  attach it to   $B(V_1, w_2)$ so that    $p(v_{k+1})_-= v_{k+1}$. Then we do foldings  to produce an irreducible graph $B(V_2, w_2)$ in which there is a  path that starts at $v_{k+1}$  and is labeled by $w_2$.  We will denote this path by $p(v_{k+1})$. Note  $B(V_1, w_2)$ is a subgraph of   $B(V_2, w_2)$. Moreover, if $B(V_1, w_2)  \underset \ne \subset   B(V_2, w_2)$, then
$\dg p(v_{k+1})_+ =1$ and $p(v_{k+1})_+ \not\in B$, hence $B(V_2, w_2)$ has  all of the desired properties and the induction step is complete.  Otherwise, we may assume  that $B(V_1, w_2)  =   B(V_2, w_2)$ and we consider two cases depending on whether or not  $p(v_{k+1})_+ \in B \subseteq  B(V_2, w_2)$.
\smallskip

First assume that $p(v_{k+1})_+ \not\in B \subseteq  B(V_1, w_2)$, that is, the path $p(v_{k+1})$ ends in a vertex of a tree $T  \subseteq F(V_1, w_2) $, where
$$
B(V_1, w_2) = B \cup F(V_1, w_2)  ,   \qquad   B \cap F(V_1, w_2) = V_1^*  ,  \quad  V_1^*  \subseteq  VB  ,
 $$
and $F(V_1, w_2)$ is a forest.  Let $q_T$ be a shortest  path in $T$
so that $(q_T)_- = p(v_{k+1})_+ $ and $(q_T)_+ \not\in B$,   $\dg (q_T)_+=1$.  We remark that
$|q_T| >0$ because, otherwise, $p(v_{k+1})_+ =  p(v_{i})_+$ for some $i, i \le k$, whence $v_{k+1} =  v_{i}$, contrary to $v_{k+1} \not\in V_k$.
Let $b \in A$ be the first letter of $\ph(q_T)$. Since $w_2$ ends in $c_\ell$, it follows that  $b \neq (c_\ell)^{-1}$. Let $b' \in A$ be different from the letters $b,  (c_\ell)^{-1}$. Then it follows from the definitions that the word
$$
w_3 \equiv  w_2  \ph(q_T) b' \equiv  w_1 c_\ell^{m+2} \ph(q_T) b'
$$
is reduced and can be used as a desired word $w$ for the set $V_{k+1} =V_{k} \cup \{ v_{k+1} \}$.

Now consider the case when $p(v_{k+1})_+ \in B \subseteq  B(V_1, w_2)$. By the definition of $m$, the vertex $p(v_{k+1})_+$ can be joined in $B(V_1, w_2)$   with a vertex $u$ of a tree $T'$,   $T'  \subseteq F(V_1, w_2)$, by a path $q$ so that $|q| \le m+1$, where $u \not\in  B$, and $u$ is connected to a vertex in $B$ by an edge of $T'$. Consider a reduced word $w_4$ equal in $F(A)$ to $ w_1 c_\ell^{m+2} \ph(q)$ which we write in the form
$$
w_4 \overset {F(A)} {=}     w_1 c_\ell^{m+2} \ph(q)  .
$$

Note that at most $| q|$ letters of the reduced word $ w_1 c_\ell^{m+2} $ could cancel with those of $\ph(q) $. Since $| q| \le m+1$, the word
$w_4$ can still be used for the set $V_1$ in place of $w_2$. As above, we construct an irreducible $A$-graph  $ B(V_2, w_4)$ with  $\ph(p(v_i)) = w_4$ for every $i =1, \dots, k+1$. The path $p(v_{k+1})$ with  $p(v_{k+1})_- = v_{k+1}$ and $\ph(p(v_{k+1})) = z_{4}$ will have its terminal vertex $p(v_{k+1})_+ $ on a tree $T$ of $ F(V_1, w_4)$. Now we can argue as in the foregoing case. We pick a shortest path $q_T$ in $T$  such that  $(q_T)_- = p(v_{k+1})_+ $,   $(q_T)_+ \not\in B$,   $\dg (q_T)_+=1$, and consider a reduced  word $w_5$ such that
$$
w_5 =  w_4 \ph(q_T) c'    \overset {F(A)} {=}   w_1 c_\ell^{m+2}  \ph(q) \ph(q_T) c' ,
$$
where $c' \in A$ is now chosen so that  $c'$ is different from the first letter of $\ph(q_T)$ and from $b^{-1}$, where $b$ is the last letter of $w_4$.   As above, we see that $|q_T| >0$.
It is straightforward to verify that the graph $B(V_2, w_5)$  has all of the required properties and Lemma~\ref{BF} is proven.
\end{proof}

We now generalize Lemma~\ref{BF} to the situation with two subsets $V_1, V_2 \subseteq  VB$.

\begin{lemma}\label{BF2} Suppose that  $B$ is a finite  irreducible $A$-graph and  no connected component of $B$ is $A$-complete. Then, for every pair of subsets $V_1, V_2 \subseteq VB$, there are nonempty reduced words $w_1 = w_1(V_1, V_2)$, $w_2= w_2(V_1, V_2)$ over $A$ and there is  an irreducible $A$-graph   $B(V_1, w_1; V_2, w_2 )$ such that $B(V_1, w_1; V_2, w_2 )$ contains $B$, $B(V_1, w_1; V_2, w_2 ) = B \cup F(V_1, w_1; V_2, w_2 )$, where $F(V_1, w_1; V_2, w_2 )$ is a forest,   $B \cap F(V_1, w_1; V_2, w_2 ) =  V_{12}^*$,   where $ V_{12}^*  \subseteq VB$,  and
$
\barr(B(V_1, w_1; V_2, w_2)) = \barr(B ) .
$

Furthermore, for every vertex $v_i \in V_i$, $i =1,2$, there is a unique path $p(v_i)$ in
$B(V_1, w_1; V_2, w_2 )$ such that $p(v_i)_- = v_i$,   $\ph(p(v_i)) \equiv  w_i$,  $p(v_i)_+$ has degree one,  $p(v_i)_+ \not\in B$, and the sets $\{ p(v_1)_+ \mid v_1 \in V_1 \}$,  $\{ p(v_2)_+ \mid v_2 \in V_2 \}$
are disjoint.
\end{lemma}

\begin{proof} First we apply Lemma~\ref{BF} to the set $V_1$ to obtain a word $w_1' \in F(A)$ and a graph $B(V_1, w_1')$ with properties of Lemma~\ref{BF}. We also apply Lemma~\ref{BF} to the set $V_2$ to obtain a word $w_2' \in F(A)$ and a graph $B(V_2, w_2')$ with properties of Lemma~\ref{BF}. Since $|A| \ge 4$, it follows that there are letters $b_1, b_2 \in A$ so that the words $w_1' b_1$, $w_2' b_2$ are reduced and $b_1 \not\in \{ b_2,  (b_2)^{-1} \}$. Let $k_2 = |w_1'| + |w_2'|$. Define the words $w_1 = w_1' b_1^{k_2}$,  $w_2 = w_2' b_2^{k_2}$. Now we attach paths $p_i(v_i)$ with $\ph(p_i(v_i)) = w_i$, over all $v_i \in V_i$, $i =1,2$, to the graph $B$ by identifying the vertices $p_i(v_i)_-$ and $v_i \in V_i \subseteq B$. Making foldings, we obtain an irreducible $A$-graph  $B(V_1, w_1; V_2, w_2)$. It is easy to see that the graph
$B(V_1, w_1; V_2, w_2)$ has all of the properties stated in Lemma~\ref{BF2}. \end{proof}

As above, let $B$ be a finite irreducible $A$-graph and $f \in A$. Let
$$
E_fB := \{ e \mid  e \in EB, \ph(e) = f  \}
$$
denote the set of all edges $e \in EB$ such that  $\ph(e) = f$ and, analogously, denote
$$
E_{f^{-1}}B := \{ e \mid e \in EB, \ph(e) =  f^{-1}  \} .
$$

Let $B_f$ denote the graph $ B \setminus (E_fB \vee   E_{f^{-1}}B )$,
 $$
 B_f := B \setminus (E_fB \vee   E_{f^{-1}}B ) ,
 $$
 and, similarly, denote
 $$
 A_f :=   A \setminus \{ f,  f^{-1} \} .
 $$
The restriction of the immersion $\ph_0 : B \to A_0$ to $B_f$ denote  $\ph_{0, f} : B_f \to A_{0, f}$, where  $V A_{0, f} = \{ 0_f \}$ and
 $E A_{0, f} = A_f$. Denote
 $$
 (E_f B)_- :=
\{ e_- \mid  e \in E_f B  \} , \qquad (  E_f B)_+ :=
\{ e_+ \mid  e \in  E_f B  \} .
 $$

 A connected component of the graph $B_f$ is called {\em $A_f$-complete } if the degree of its every vertex is equal to $|A_f| = 2n-2$. Otherwise, a connected component of the graph $B_f$ is called {\em $A_f$-incomplete}.

 Let $B_f^{\textsf{com}}$ denote the set of all  $A_f$-complete  connected  components of  $B_f$ and $B_f^{\textsf{inc}}$ denote the set of all  $A_f$-incomplete  connected  components of  $B_f$. We also denote
  $$
 V_1^{\textsf{inc}}(f_-) :=   B_f^{\textsf{inc}} \cap (E_f B)_- , \qquad
 V_2^{\textsf{inc}}(f_+) :=   B_f^{\textsf{inc}} \cap (E_f B)_+ .
  $$

 Let us apply Lemma~\ref{BF2} to the irreducible $A_f$-graph  $B_f^{\textsf{inc}}$ and its sets of vertices $  V_1^{\textsf{inc}}(f_-) $ and   $V_2^{\textsf{inc}}(f_+)$. As a result, we obtain two reduced words $w_1, w_2$ over $A_f$ and an irreducible $A_f$-graph
 $$
 B_f^{\textsf{inc}}( V_1^{\textsf{inc}}(f_-), w_1; V_2^{\textsf{inc}}(f_+), w_2)
 $$
 with the properties of  Lemma~\ref{BF2}.
\smallskip

 Now we construct more labeled $A$-graphs in the  following fashion. For every edge $e \in E_f B$, we consider a new graph which is a path $p_1(e) e p_2(e)^{-1}$ so that $\ph(p_1(e) ) \equiv w_1$,  $\ph(e ) = f$,  $\ph(p_2(e) )  \equiv w_2$. We replace every $e \in E_f B$ in $B$ by such a path $p_1(e) e p_2(e)^{-1}$ and then do edge foldings to get an  irreducible $A$-graph which we denote $B[f]$. We will call  this graph $B[f]$, obtained by the described application of Lemma~\ref{BF2},  an {\em $f$-treatment} of $B$. It is clear that the graph $B[f]$ can be thought of as the image of $B$ under the automorphism of the free group $F(A)$ that takes $f$ to $w_1 f w_2^{-1}$ and takes $a$ to $a$ if $a \in A_f$. We also observe that the  connected components of $B_f$ and those of $B[f]_f$ are in the natural bijective correspondence. Moreover, it follows from the definitions and Lemma~\ref{BF2} that a connected component of $B_f$  is $A_f$-complete if and only if its image in $B[f]_f$ is  $A_f$-complete and, if they are both $A_f$-complete, then they are identical.

 Suppose $c \in A, c \not\in \{f, f^{-1} \}$. Using the notation analogous to what we introduced above for $f$, we consider the graph
 $$
 B[f]_c := B[f] \setminus (E_cB[f] \vee E_{c^{-1}} B[f]) .
 $$
 As above, we consider a partition
 $$
 B[f]_c = B[f]_c^{\textsf{inc}} \vee B[f]_c^{\textsf{com}}  ,
 $$
where $B[f]_c^{\textsf{inc}} $ is the union of $A_c$-incomplete connected components of $B[f]_c$ and $ B[f]_c^{\textsf{com}}$ is the union of $A_c$-complete connected  components of $B[f]_c$.

 \begin{lemma}\label{CAB} Suppose $B$ is  a finite irreducible $A$-graph, $|A| \ge 6$,  no connected component of $B$ is a complete $A$-graph, $f \in A$, and $B[f]$ is an $f$-treatment of $B$. Then, for every $c \in A \setminus  \{ f, f^{-1} \}$, considering the sets $V(   B[f]_c^{\textsf{com}} )$, $V(B[f]_f^{\textsf{com}} )$ as subsets of $V B[f]$, one has  that
 $$
 V(   B[f]_c^{\textsf{com}} ) {\subseteq}  V(B[f]_f^{\textsf{com}} )
 $$
 and either $V(   B[f]_c^{\textsf{com}} )= V(B[f]_f^{\textsf{com}} ) = \varnothing$ or $V(   B[f]_c^{\textsf{com}} ) \ne V(B[f]_f^{\textsf{com}} )$.   Moreover,
$$ | V(   B[f]_f^{\textsf{com}} ) | =  | V(   B_f^{\textsf{com}} ) |  . $$
 \end{lemma}

\begin{proof} As was observed above, connected components of $B_f$ and those of $B[f]_f$ are in the natural bijective correspondence, moreover, a connected component $C$ of $B_f$  is $A_f$-complete if and only if its image in $B[f]_f$ is  $A_f$-complete and, if, they are both $A_f$-complete, then they are identical. Hence, the $A_f$-graphs
$B_f^{\textsf{com}} $ and  $B[f]_f^{\textsf{com}} $ are isomorphic and
$ | V(   B[f]_f^{\textsf{com}} ) | =  | V(   B_f^{\textsf{com}} ) |$.

Now suppose $v \in  V (B[f]_c^{\textsf{com}})$. Considering $v$ as a vertex of $B[f]$, we see that there are edges $e_1, e_2$ in $B[f]$ starting at $v$ such that  $\ph(e_1) = \ph(e_2^{-1}) = f$. However, it follows from
Lemma~\ref{BF2} that no vertex $u$ of $B[f]$  has distinct edges $e_1, e_2$ starting at $v$ such that $\ph(e_1) = \ph(e_2^{-1}) = f$.  This  shows that  $v \in V( B[f]_f^{\textsf{com}} )$, whence,
$V(   B[f]_c^{\textsf{com}} ) {\subseteq}  V(B[f]_f^{\textsf{com}} ) $, as desired.

Finally, assume that    $ V( B[f]_c^{\textsf{com}} ) = V( B[f]_f^{\textsf{com}} ) \ne \varnothing$. Then for every vertex $v \in V( B[f]_c^{\textsf{com}} )  {\subseteq}  V(B[f])    $ we have $\dg v = 2n-2$ in both graphs $B[f]_f^{\textsf{com}}$ and $ B[f]_c^{\textsf{com}} $. Hence, $\dg v = 2n$ in $B[f]$ and so $ B[f]^{\textsf{com}}  $ contains all of the vertices of $ B[f]_c^{\textsf{com}} \neq \varnothing$. This, however, means
that $B[f]$ and, hence, $B$ contains an $A$-complete connected component. This contradiction to  the assumption  that $B^{\textsf{com}}$  is empty proves that  $V(   B[f]_c^{\textsf{com}} ) \underset {\ne}{\subset}  V(B[f]_f^{\textsf{com}} )$
whenever $ V(B[f]_f^{\textsf{com}} )$ is not empty, as desired.
 \end{proof}

  \begin{lemma}\label{REF} Suppose $B$ is  a finite irreducible $A$-graph, $|A| \ge 6$,  and no connected component of $B$ is a complete $A$-graph. Then there exists a finite sequence of $f_1$-, $\dots, f_{\ell+1}$-treatments of the graph $B$  so that the resulting graph, denoted  $B[f_1, \dots, f_{\ell+1}]$, has the following property. For $f \in A$, denote
  $$
B[f_1, \dots, f_{\ell+1}]_f := B[f_1, \dots, f_{\ell+1}] \setminus  (E_f(B[f_1, \dots, f_{\ell+1}]) \vee E_{f^{-1}}(B[f_1, \dots, f_{\ell+1}])  ) .
$$
For each $f \in A$,  every connected component of $B[f_1, \dots, f_{\ell+1}]_f $ is   $A_f$-incomplete. In addition, for every edge $e \in E(B[f_1, \dots, f_{\ell+1}])$ such that $\ph(e)= f_{\ell+1}$, one has
$\dg e_-= $ $\dg  e_+= 2$  and, if $h_1(e) e h_2(e)$ is a reduced path
in $B[f_1, \dots, f_{\ell+1}]$, where $h_1(e)$, $h_2(e) $ are edges,
then the labels $\ph(h_1(e))$, $\ph(h_2(e))$  are independent of $e$.
\end{lemma}

\begin{proof}  Pick an arbitrary letter $f_1 \in A$ and do an $f_1$-treatment of the graph $B$.
Assuming that $B[f_1]_{f_1}^{\textsf{com}} \ne \varnothing$, we can see from
Lemma~\ref{CAB} that  the result is a new irreducible $A$-graph $B[f_1]$ such that, for every $c \in A$, where $c \not\in \{ f_1,  f_1^{-1} \}$, we have
$$
 V(   B[f_1]_c^{\textsf{com}} ) \underset {\ne}{\subset}  V(B[f_1]_{f_1}^{\textsf{com}} ) , \qquad  | V(   B[f_1]_{f_1}^{\textsf{com}} ) | =  | V(   B_{f_1}^{\textsf{com}} ) |  .
$$

Picking an edge $f_2$ such that $f_2 \in A$, $f_2 \not\in \{ f_1, f_1^{-1} \}$, we will do a second $f_2$-treatment of the graph $B[f_1]$ and get a new irreducible $A$-graph $B[f_1, f_2] := B[f_1][f_2] $.

Assuming that $B[f_1, f_2]_{f_2}^{\textsf{com}} \ne \varnothing$, we obtain from
Lemma~\ref{CAB} that, for every $c \in A$ such that $c \not\in \{ f_2, f_2^{-1} \}$, we have
\begin{gather}\notag
 V(   B[f_1,f_2]_c^{\textsf{com}} ) \underset {\ne}{\subset}  V(B[f_1, f_2]_{f_2}^{\textsf{com}} )  ,  \\ \notag
 | V(   B[f_1, f_2]_{f_2}^{\textsf{com}} ) | = | V(   B[f_1]_{f_2}^{\textsf{com}} ) | <  | V(   B[f_1]_{f_1}^{\textsf{com}} ) | = | V(   B_{f_1}^{\textsf{com}} ) |  .
\end{gather}

Then we pick an edge $f_3$, where $f_3 \in A$, $f_3 \not\in \{ f_2, f_2^{-1} \}$, and do a third $f_3$-treatment of the graph $B[f_1, f_2]$ to construct  an irreducible $A$-graph $B[f_1, f_2, f_3] = B[f_1, f_2][f_3] $ with further decreased $ | V(   B[f_1, f_2, f_3]_{f_3}^{\textsf{com}} ) |$ and so on.  Assuming for each $j =1, \dots, i$ that $B[f_1, \ldots, f_j]_{f_j}^{\textsf{com}} \ne \varnothing$, we obtain from
Lemma~\ref{CAB} that,  for every $c \in A$ such that $c \not\in \{ f_j, f_j^{-1} \}$, it is true that
\begin{gather*}
 V(   B[f_1, \ldots, f_j]_c^{\textsf{com}} ) \underset {\ne}{\subset}  V(B[f_1, \ldots, f_j]_{f_j}^{\textsf{com}} )  , \\
 | V(   B[f_1, \ldots, f_j]_{f_j}^{\textsf{com}} ) | < | V(   B[f_1, \ldots, f_{j-1}]_{f_{j-1}}^{\textsf{com}} ) |  , \\
 | V(   B[f_1, \ldots, f_j]_{f_j}^{\textsf{com}} ) | = | V(   B[f_1, \ldots, f_{j-1}]_{f_j}^{\textsf{com}} ) |  .
\end{gather*}
Hence, it follows from these equalities and inequalities that
\begin{align*}
| V(   B[f_1, \ldots, f_{i+1}]_{f_{i+1}}^{\textsf{com}} ) | & < \ldots <
| V(   B[f_1, \ldots, f_{j}]_{f_{j}}^{\textsf{com}} ) | < \ldots  \\ & < | V(   B[f_1, f_2]_{f_2}^{\textsf{com}} ) | <  | V(   B[f_1]_{f_1}^{\textsf{com}} ) | .
\end{align*}

Therefore, the graph $ B[f_1,\dots, f_{i+1}]^{\textsf{com}}_{f_{i+1}} $ will eventually become empty and then \\ $B[f_1,\dots, f_{i+1}]_c^{\textsf{com}}$ will be empty for every  $c \in A$.
Thus we may suppose  that $ B[f_1,\dots, f_{\ell}]_{f_{\ell}}^{\textsf{com}} $ is empty for some $\ell \ge 1$.

We will do one more $f_{\ell+1}$-treatment of the graph $B[f_1,\dots, f_{\ell} ]$. Pick an edge $f_{\ell+1}$ such that  $f_{\ell+1} \in A$, $f_{\ell+1} \not\in \{ f_\ell, f_\ell^{-1} \}$,  do an $f_{\ell+1}$-treatment of the graph $B[f_1, \dots, f_{\ell}]$ and obtain an irreducible $A$-graph $B[f_1, \dots, f_{\ell+1}] := B[f_1, \dots, f_{\ell}  ][f_{\ell+1}] $. It follows from Lemma~\ref{CAB} that, for every $c \in A$, the set $B[f_1,\dots, f_{\ell+1}]_c^{\textsf{com}}$ is still empty, in particular,
$B[f_1,\dots, f_{\ell+1}]_f^{\textsf{com}} = \varnothing$. In addition, by the definitions and Lemma~\ref{BF2}, we will also have the claimed property of  edges $e \in EB[f_1, \dots, f_{\ell+1} ]$ such that $\ph(e)= f_{\ell+1}$.
\end{proof}

\begin{lemma}\label{wd4} Suppose $|A| \ge 4$, $d_1, d_2 \in A$ and $w$ is a nonempty reduced word over $A$. Then there are letters $b_1, b_2 \in A$ such that the word
$d_1 (b_1 w b_2)^2 d_2$ is reduced. Furthermore, let $c \in A$, $c \not\in  \{ b_2,  b_2^{-1} \}$, and $n_w := |w| +1$. Then the word
\begin{gather}\label{w4}
z_4(w) := b_1 w b_2 c^{n_w +1} b_2 \,   b_1 w b_2 c^{n_w+2} b_2 \,  b_1 w b_2 c^{n_w+3} b_2 \,  b_1 w b_2 c^{n_w+4} b_2
\end{gather}
has the following properties.  If  $(b_2 c^{n_w+i} b_2)^k$, where $k = \pm 1$, $i =1,2,3,4$, is a subword of $z_4(w)$, then $k = 1$ and the location of the subword $b_2 c^{n_w+i} b_2$ is {\em standard}, i.e., it is the suffix of length $n_w+i+2$ of the $i$th syllabus $b_1 w b_2 c^{n_w+i} b_2$ of
 $z_4(w)$. Moreover, the word $d_1 z_4(w) d_2$ is reduced.
\end{lemma}

\begin{proof} Since  $|A| \ge 4$, there are at least two letters $x \in A$ such that  the word $d_1 x w$ is reduced. Also, there are at least two letters $y \in A$ such that the word $w y d_2$ is reduced. If $d_1 (x w y)^2 d_2$ is not reduced  then $x =  y^{-1}$. For every $x$ with  $d_1 x w$ being  reduced,  there is at most one $y$ with $w y d_2$ being reduced such that $d_1 (x w y)^2 d_2$ is not reduced. Hence there exist $x$ and $y$
in $A$ such that $d_1 (x w y)^2 d_2$ is  reduced. Denote $b_1 := x  $ and $b_2 := y$.

Since  $c \not\in  \{ b_2, b_2^{-1} \}$  and the maximal power  of $c$ that can occur in $b_1 w b_2$ is $c^{|w| +1} = c^{n_w}$, it follows from the definition \eqref{w4} of the word  $z_4(w)$ that $b_2 c^{n_w+i} b_1$  can only occur in $z_4(w)^{\pm 1}$   as the suffix of length $n_w+i+2$ of the $i$th syllabus $b_1 w b_2 c^{n_w+i} b_2$ of  $z_4(w)$.

It remains to note that  the word $d_1 z_4(w) d_2$ is reduced because the word  $d_1 (x w y)^2 d_2$ is  reduced.
\end{proof}

\section{Proof of Theorem~\ref{Th1}}

As in Sect. 2, consider  the graphs $W$, $X, Y$, $Z$.

\begin{lemma}\label{wd42}
Suppose either graph $Q$,  where   $Q \in \{X, Y\}$,  contains a path $p_Q$ so that every vertex of $p_Q$ has degree 2 in $Q$, $ \beta_Q (p_Q) \equiv  z_4(w)$, where $w$ is a nonempty reduced word  over the alphabet $A = E Z$  and $z_4(w)$ is the word defined by \eqref{w4}. Let $z_4(w) = z_{41} z_{42}$ be a factorization of  $z_4(w)$  so that
\begin{gather*}
 z_{41} \equiv  b_1 w b_2 c^{n_w+1} b_2 b_1 w b_2 c^{n_w+2} b_2  , \quad
 z_{42} \equiv  b_1 w b_2 c^{n_w+3} b_2 b_1 w b_2 c^{n_w+4} b_2
\end{gather*}
and let $v_2(p_Q)$ be a vertex of the path $p_Q$ that defines a corresponding  factorization
$p_Q = p_{Q1}p_{Q2}$ so that $\ph(p_{Q1}) \equiv z_{41}$, $\ph(p_{Q2}) \equiv z_{42}$.

Furthermore, assume that there exists a vertex $u \in VW$, where $W = \core( X  \times_Z Y)$, such that either $\alpha_X(u) = v_2(p_X)$ and  $\alpha_Y(u) \in  p_Y$ or $\alpha_X(u) \in  p_X$ and  $\alpha_Y(u)  = v_2(p_Y)$. Then  $\alpha_X(u) = v_2(p_X)$,  $\alpha_Y(u)  = v_2(p_Y)$, and there is a path $p$ in $W$ such that $\alpha_X(p) =p_X$ and $\alpha_Y(p) =p_Y$.
\end{lemma}

\begin{proof} For definiteness, suppose $u \in VW$ is such  that  $\alpha_X(u) = v_2(p_X)$ and  $\alpha_Y(u) \in  p_Y$. Let $p_Y = p_{Y1}p_{Y2}$ be the factorization
of $p_Y$ defined by the vertex  $\alpha_Y(u)$.  Pick a path  $p_{Yi^*}$, $i^* = 1,2$, such that
 $ | p_{Yi^*} | \ge \tfrac 12| p_{Y} |$. Since the   factorization of $p_X$ defined by the vertex  $\alpha_X(u) = v_2(p_X) $ defines  the   factorization $z_4(w) = z_{41} z_{42}$ of the word $z_4(w) =   \beta_X(p_X) $, it follows that if $i^* = 2$ then the word $  \beta_Y(p_{Y2}) $ contains a subword
 $b_2 c^{n_w+3} b_2$   or contains a subword $b_2^{-1} c^{-n_w-2} b_2^{-1}$. The second subcase, however, is impossible by Lemma~\ref{wd4}. Similarly, if $i^* = 1$ then the word $\gamma \beta_Y(p_{Y1}) $ contains a subword
 $b_2 c^{n_w+2} b_2$  or contains a subword $b_2^{-1} c^{-n_w-3} b_2^{-1}$.  The second subcase  is impossible by Lemma~\ref{wd4}.  In either case $i^* = 1,2$, we can apply Lemma~\ref{wd4} to the subword
 $b_2 c^{n_w+3} b_2$ or the subword  $b_2 c^{n_w+2} b_2$ of  $ \beta_Y(p_{Y}) $ and obtain the desired conclusion.
 \end{proof}

We now define a special type of transformations over the graphs
$W, X, Y$, $Z$, called $(f, p)$-transformations.

Suppose that $f$ is an edge of the graph $Z$ and $p = e_1 \dots e_\ell$ is a path
in $Z$ such that  $p_- = (e_1)_- = f_-$,  $p_+ = (e_\ell)_+ = f_+$, and
 there are no occurrences of $f$, $f^{-1}$ among  edges $e_1, \dots, e_\ell$.

Let  $Q \in \{W,  X, Y, Z \}$ and $\ph : Q \to Z=U$
be a canonical immersion.
Consider a set $H_f$ of all edges $h$ in $EW \vee EX \vee EY \vee EZ$
that are sent by  $\ph$ to $f$. We replace every edge  $h \in H_f$    with a path
$p(h) = e_1(h) \dots e_\ell(h)$
such that
$$
\ph(e_i(h)) := \ph(e_i)
$$
for every  $i=1,\ldots, \ell$,
and we extend  every  map  $\nu \in \{ \alpha_X, \dots , \beta_Y \}$
to paths  $p(h)$ so that if $\nu : Q \to Q'$, where $Q, Q' \in \{ W, X, Y, Z \}$, and
$\nu(h) = h'$, then we extend  $\nu$ to paths   $p(h)$, $p(h')$ by setting
$
\nu(e_1(h)) :=  e_1(h'), \  \dots,   \  \nu(e_\ell(h)) :=  e_\ell(h') .
$
Thus obtained graphs we denote by $Q_{fp}$, where $Q \in \{ W, X, Y, Z \}$,
and thus obtained  maps we denote by  $ \alpha_{X_{fp}}, \dots,  \beta_{Y_{fp}} $.

We now perform  folding process over every modified graph $Q_{fp}$,
$Q \in \{ W, X, Y, Z \}$. Recall that this is an inductive procedure
which identifies every pair of oriented edges $g_1, g_2 \in E Q_{fp}$ whenever $(g_1)_- =(g_2)_-$  and
$g_1, g_2$  have the same  labels $\ph(g_1)$, $ \ph(g_2)$.

Note that folding process  decreases by one the reduced rank  $\rr (Z)$ because the new path
$p(f) = e_1(f) \dots e_\ell(f)$ that replaces the edge $f$ in $Z$ will be attached to the
path  $p = e_1 \dots e_\ell$  in $Z \setminus \{ f , f^{-1}\} $ thus producing a  graph $\bar Z_{fp}$ with
$\rr (\bar Z_{fp}) =  \rr (Z)-1$.
Clearly, $\bar Z_{fp}$ is a bouquet of $\rr (Z)-1$ circles and $\core(\bar Z_{fp}) = \bar Z_{fp}$.

Furthermore, when a folding process applied to $Q_{fp}$,   $Q \in \{ W,  X, Y, Z \}$,
is complete and produces a graph  $\bar Q_{fp}$,  we will take the  core of  $\bar Q_{fp}$ thus obtaining
a graph  $Q_{p} = \core(\bar Q_{fp})$.   It follows from the definitions that we will have  maps
$ \alpha_{X_{p}}, \dots , \beta_{Y_{p}}$ with
 properties of original maps $\alpha_X, \ldots ,  \beta_Y$.
This alteration of the  graphs $W,  X, Y, Z $  by means of an edge
$f \in EZ$ and a path $p$ in $Z$ will be called an {\em $(f, p)$-transformation}  over  the graphs
$W, X, Y$, $Z$.

We will say that an $(f, p)$-transformation is {\em conservative}
if it preserves the numbers $\barr (X)$, $\barr (Y)$,
$\barr (W_s)$ for every connected component $W_s$ of $W$,  $s \in S(H,K)$,  and
the core $\core(X_{p}  \times_{Z_{p}}  Y_{p} )$ of the pullback
$X_{p}  \times_{{Z_{p}}}  Y_{p}$
coincides with  the graph $W_{p}$.
Thus a conservative $(f, p)$-transformation decreases  $\barr (Z)$ by one while
keeping the numbers $\barr (X)$, $\barr (Y)$, $\barr (W)$ unchanged.

Here is our principal technical result that will be used to prove Theorem~\ref{Th1}.

\begin{lemma}\label{cntr}
Suppose $\barr (Z) \ge 2$, $U = Z$, $\gamma = \mbox{id}_Z$,    and neither of the maps
$ \beta_X : X \to Z$, $\beta_Y : Y \to Z$ is a covering.
Then  there is an automorphism $\tau$ of the free group $F(EZ) = \pi_1(Z)$ such that
Stallings graphs of subgroups $\tau(H)$, $\tau(K)$, $\tau( \bigvee_{s\in S(H,K)} (H \cap sKs^{-1}) )$, denoted
$X^\tau$, $Y^\tau$, $W^\tau$, resp., have no vertices of degree 1 and there exists a  conservative $(f,p)$-transformation over the  graphs $W^\tau$, $X^\tau$, $Y^\tau$, $Z^\tau = Z$.
\end{lemma}

 \begin{proof}  Assume that neither of the maps $ \beta_X : X \to Z$,  $ \beta_Y : Y \to Z$ is a covering.
 Consider the disjoint union $X \vee Y$ of $X, Y$ as a graph $B$ and the set $EZ$ as the alphabet $A$.
   Let  $\ph : EB \to A$ be defined on $EX$ as the restriction of $\beta_X$ and on $EY$ as the restriction of  $\beta_Y$. We also consider $W$ as an $A$-graph by using the map $\ph : EW \to A$, where $\ph$ is the   restriction of  $\beta_X \alpha_X =  \beta_Y \alpha_Y$.

 Since neither of $X, Y$ is $A$-complete, we can apply Lemma~\ref{REF}
 and find a sequence of $f_1$-, $\dots, f_{\ell+1}$-treatments for $B$ which transform  the graph $B = X \vee Y$ into $B[f_1, \dots, f_{\ell+1}]$  with the properties of Lemma~\ref{REF}. Note that an $f_{}$-treatment of $B = X \vee Y$ can  equivalently be  described as an application of a suitable automorphism $\tau_f$ of the free group $F(A) = \pi_1(Z)$.
 We specify that this automorphism  $\tau_f$  is the composition of an automorphism $\theta_{f}$, given by
$\theta_{f}(f) = w_1 fw_2^{-1}$, $\theta_{f}(a) = a$ for every $a \in A_f$,  where $w_1, w_2$ are some words
over $A_f$, and an inner automorphism of $F(A)$ applied, if necessary, to move the base vertices of graphs
$\theta_{f}(X)$, $\theta_{f}(Y)$, representing subgroups  $\theta_{f}(H)$, $\theta_{f}(K)$, resp.,  so that
the base vertices, after the move,  would be in $\core(\theta_{f}(H))$, $\core(\theta_{f}(K))$ and the correspondence between graphs and subgroups would be preserved.
Therefore, the composition of these
 $f_1$-, $\dots, f_{\ell+1}$-treatments can be induced by a suitable  automorphism of the free group $F(A)$  applied
to subgroups $H, K$, $\bigvee_{s\in S(H,K)} (H \cap sKs^{-1})$ and to corresponding Stallings graphs
$X, Y, W$.
By Lemmas~\ref{REF}, \ref{BF2}, we may assume  that, for  every $c \in A$, the graph $B[f_1, \dots, f_{\ell+1}]_c$  has no $A_c$-complete component and that  every edge $e \in EB[f_1, \dots, f_{\ell+1}]$ such that $\ph(e) = f_{\ell+1}$ is contained in a reduced path  $h_1(e) e h_2(e)$, where  $h_1(e), h_2(e)$ are edges,
  so that  $\dg  e_-= $ $\dg  e_+ = 2$  and the letters $\ph(h_1(e))$, \ $\ph(h_2(e))$ are independent of  $e$.

Denote
 \begin{gather}\label{a12}
 \ph(h_1(e)) = d_1  , \quad  \ph(h_2(e)) = d_2  ,
\end{gather}
where $d_1, d_2 \in A_{f_{\ell+1}} =  EZ \setminus  \{f_{\ell +1}, f_{\ell +1}^{-1} \}$.  Since $B = X \vee Y$, we can represent the graph
$B[f_1 , \dots, f_{\ell +1}]$ in the form
$$
B[f_1 , \dots, f_{\ell +1}] =  X[f_1 , \dots, f_{\ell +1}] \vee Y[f_1 , \dots, f_{\ell +1}] ,
$$
where $X[f_1 , \dots, f_{\ell +1}] = X^\tau$ is obtained from $X$ by these $f_1$-, $\dots$, $f_{\ell+1}$-treatments
and $Y[f_1 , \dots, f_{\ell +1}]= Y^\tau$ is obtained from $Y$ by the $f_1$-, $\dots$, $f_{\ell+1}$-treatments.
Similarly, let $W[f_1 , \dots, f_{\ell +1}] = W^\tau$ denote the graph obtained from $W$ by the $f_1$-, $\dots$, $f_{\ell+1}$-treatments.
\smallskip

Consider  the irreducible $A_{f_{\ell+1}}$-graph
 $$
 B[f_1 , \dots, f_{\ell +1}]_{f_{\ell+1}}   = X[f_1 , \dots, f_{\ell +1}]_{f_{\ell+1}}  \vee Y[f_1 , \dots, f_{\ell +1}]_{f_{\ell+1}}
 $$
and apply Lemma~\ref{BF} to this graph and to the vertex set $V_1 = VB[f_1 , \dots, f_{\ell +1}]_{f_{\ell+1}} $. According to
Lemma~\ref{BF}, there exists a reduced nonempty word $w$ over the alphabet
$A_{f_{\ell+1}} = EZ \setminus \{f_{\ell +1}, f_{\ell +1}^{-1} \}$ with the properties of Lemma~\ref{BF}. In particular, for every vertex $v \in VB[f_1 , \dots, f_{\ell +1}]_{f_{\ell+1}} $,  the $A_{f_{\ell+1}}$-graph $B[f_1 , \dots, f_{\ell +1}]_{f_{\ell+1}}$ contains no path $p$ such that $p_- = v$ and $\ph(p) \equiv  w$.

Since $A = EZ$ and $|A| \ge 6$, we have $| A_{f_{\ell+1}} | \ge 4$. Hence, Lemma~\ref{wd4} applies to the word $w$, to the alphabet $A_{f_{\ell+1}} $ and to the letters $d_1, d_2 \in A_{f_{\ell+1}} $ defined by equalities \eqref{a12}. By Lemma~\ref{wd4}, there are letters $b_1, b_2 \in A_{f_{\ell+1}}$ such that the word $d_1(b_1 w b_2)^2 d_2$ is reduced.  Furthermore, let $c \in A_{f_{\ell+1}}$, $c \not\in \{ b_2, b_2^{-1} \}$, and $n_w = |w| +1$. Then the word $z_4(w)$ over $A_{f_{\ell+1}}$, given by formula \eqref{w4},  has the properties stated in Lemma~\ref{wd4}.

Now we perform an $(f_{\ell+1}, z_4(w))$-transformation  over the graphs
$$
W[f_1 , \dots, f_{\ell +1}]  = W^\tau , \quad X[f_1 , \dots, f_{\ell +1}] = X^\tau, \quad   Y[f_1 , \dots, f_{\ell +1}] = Y^\tau, \quad  Z .
$$
Making this  transformation turns every edge
$e$ such that  $\ph(e)  = f_{\ell+1}$ into a path $p = p(e)$ such that
$p_- = e_-$, $p_+ = e_+$, and $\ph(p) \equiv z_4(w)$. Let
$B (f_{\ell+1}, z_4)$, $X (f_{\ell+1}, z_4)$, $Y (f_{\ell+1}, z_4)$,  $W (f_{\ell+1}, z_4)$ denote the resulting $A_{f_{\ell+1}}$-graphs.

Observe that if $\mu : F(A) \to F(A_{f_{\ell+1} })$ is the epimorphism defined by
$\mu(b) = b$ for every $b \in A_{f_{\ell+1} }$ and $\mu(f_{\ell+1}) =  z_4(w)$ then we have
\begin{gather*}
W(f_{\ell+1}, z_4) = \mu( W[f_1 , \dots, f_{\ell +1}] ) =  \mu(W^\tau) , \\
X(f_{\ell+1}, z_4) = \mu( X[f_1 , \dots, f_{\ell +1}] ) =  \mu(X^\tau) , \\
Y(f_{\ell+1}, z_4) = \mu( Y[f_1 , \dots, f_{\ell +1}] ) =  \mu(Y^\tau) .
\end{gather*}

A path $p = p(e)$ in graphs $X (f_{\ell+1}, z_4)$, $Y (f_{\ell+1}, z_4)$,  $W (f_{\ell+1}, z_4)$ such that  $\ph(p)  \equiv  z_4(w)$  and $p$ results from an edge $e$ with $\ph(e)  =    f_{\ell+1}$ will be called a  {\em standard $z_4$-path}.  Let $p$ be a standard $z_4$-path and $p = p_1 p_2 p_3 p_4$ be a factorization of $p$ so that $\ph(p_i)  \equiv b_1 w b_2 c^{n+i} b_2$, where $i =1,2,3,4$. Denote $v_2(p) := (p_2)_+$.

Recall that $Z$ has a single vertex and an $(f_{\ell+1}, z_4(w))$-transformation  converts $Z$ into $Z_{f_{\ell+1}} =  Z \setminus  \{   f_{\ell+1},  f_{\ell+1}^{-1} \}$  such that $E Z_{f_{\ell+1}} = A_{f_{\ell+1}}$. Let
$$
\ph_{f_{\ell+1}} :  B(f_{\ell+1}, z_4) = X (f_{\ell+1}, z_4)  \vee Y(f_{\ell+1}, z_4)   \to   Z_{f_{\ell+1}}
$$
denote the corresponding immersion.  Consider the pullback
$$
X(f_{\ell+1}, z_4)     \times_{Z_{f_{\ell+1}} }   Y(f_{\ell+1}, z_4)
$$
and its core $ \widetilde W  :=  \core (X(f_{\ell+1}, z_4)     \times_{Z_{f_{\ell+1}} }
Y(f_{\ell+1}, z_4))$.

Let us prove the equality $\widetilde W = W (f_{\ell+1}, z_4)$.

Let
$$
\widetilde \alpha_X :  \widetilde W \to  X(f_{\ell+1}, z_4) , \quad
\widetilde \alpha_Y :  \widetilde W \to  Y(f_{\ell+1}, z_4)
$$
denote the projection maps. Suppose $u \in V \widetilde W$ is a vertex such that $\widetilde \alpha_X(u) = v_2(p_X)$, where $p_X$ is a standard $z_4$-path of $X (f_{\ell+1}, z_4)$.

First we assume that
\begin{gather}\label{asmY}
\widetilde \alpha_Y(u) \in  Y[f_1 , \dots, f_{\ell +1}]_{f_{\ell+1}}  ,
\end{gather}
here the graph $ Y[f_1 , \dots, f_{\ell +1}]_{f_{\ell+1}}$ is regarded as a subgraph of $Y(f_{\ell+1}, z_4)$.

Note that the conclusion of Lemma~\ref{BF} holds for the graph  $B [f_1 , \dots, f_{\ell +1}]_{f_{\ell+1}}$, for the set $V_1 = V B [f_1 , \dots, f_{\ell +1}]_{f_{\ell+1}}$ and for the reduced word $w_1 \equiv b_1 w$ in place of $w$. Hence, the graph  $ Y[f_1 , \dots, f_{\ell +1}]_{f_{\ell+1}}$ contains no path $r$ such that $r$ starts at $ \widetilde \alpha_Y(u) $ and $r$ has  the label $\ph_{f_{\ell+1}}(r) \equiv  b_1 w$.
Since there is a path $p'$ in $X (f_{\ell+1}, z_4)$ such that $p'$  starts at $ \widetilde \alpha_X(u) = v_2(p_X)$, $\ph_{f_{\ell+1}}(p') \equiv  b_1 w b_2 c^{n+3} b_2$, and all vertices of $p'$ have degree 2, it follows that there is a path $q$ in $ Y (f_{\ell+1}, z_4) $  such that $q$  starts at $ \widetilde \alpha_Y(u)$ and
$$
\ph_{f_{\ell+1}}(q) \equiv  \ph_{f_{\ell+1}}(p') \equiv    b_1 w b_2 c^{n_w+3} b_2 .
$$

Let $q = q_1 q_2$ be the  factorization of $q$ defined so that
$$
\ph_{f_{\ell+1}}(q_1) \equiv b_1 w  , \quad  \ph_{f_{\ell+1}}(q_2) \equiv b_2 c^{n_w+3} b_2 .
$$
By the above observation based on Lemma~\ref{BF}, the path $q_1$ may not be entirely contained in $Y[f_1 , \dots, f_{\ell +1}]_{f_{\ell+1}}  \subseteq  Y (f_{\ell+1}, z_4)$.
On the other hand, if $(q_1)_+ \in  Y [f_1 , \dots, f_{\ell +1}]_{f_{\ell+1}}  \subseteq  Y (f_{\ell+1}, z_4) $,
then, in view of   \eqref{asmY}, the path $q_1$ would have to contain a standard $z_4$-path of $ Y (f_{\ell+1}, z_4)$ which is impossible for $|q_1| <  |z_4(w)|$.
Therefore,    $(q_1)_+ \not\in  Y [f_1 , \dots, f_{\ell +1}]_{f_{\ell+1}}  \subseteq  Y (f_{\ell+1}, z_4) $  and  the vertex $(q_1)_+$ must belong to a standard $z_4$-path $p_Y$ of $ Y (f_{\ell+1}, z_4)$.

Since $|q_1| = |w|  +1$,  $|q_2| = |w|  +6$, and $|p_Y| = |z_4(w)|= 8|w| +26$,
it follows from   \eqref{asmY} and from $(q_1)_+ \in p_Y$  that  $q_2$ is a subpath of $p_Y^{\pm 1}$ and  there is a factorization  $p_Y = p_{Y1} p_{Y2} $ defined by the vertex $(q_2)_+ \in  p_Y$, where a shortest path out of $p_{Y1}, p_{Y2}$ contains $q_2^{\pm 1}$ and
\begin{gather}\label{minpy}
\min(| p_{Y1} |, | p_{Y2}| ) \le |q| = 2|w| +7  .
\end{gather}
On the other hand, if $p_Y = p_{Y3} p_{Y4}  =  p_{Y5} p_{Y6}$ are factorizations of $p_Y$ so that the words $\ph_{f_{\ell+1}}(p_{Y3})$, \ $\ph_{f_{\ell+1}}(p_{Y6})$ contain the standard occurrence of  the word $b_2 c^{n_w+3} b_2$ in $z_4(w)$, then
$| p_{Y3} | \ge 6|w| +18$ and $| p_{Y6} | \ge 3|w| +14$. These inequalities, in view of \eqref{minpy}, mean that the occurrence of
$$
(b_2 c^{n_w+3} b_2)^{\pm 1} = \ph_{f_{\ell+1}}(q_2)^{\pm 1}
$$
in $z_4(w) \equiv  \ph_{f_{\ell+1}}(p_{Y})$ is not standard. This contradiction  to Lemma~\ref{wd4}  proves that the inclusion \eqref{asmY} is impossible.

Thus it is shown that, for every  vertex $u \in V \widetilde W$,
if  $\widetilde \alpha_X(u) = v_2(p_X)$ for some standard $z_4$-path  $p_X$ of $X (f_{\ell+1}, z_4)$, then  $\widetilde \alpha_Y(u) \in p_Y$, where $p_Y$ is a standard  $z_4$-path in $Y (f_{\ell+1}, z_4)$.  Switching the graphs $X (f_{\ell+1}, z_4)$ and $Y (f_{\ell+1}, z_4)$ in the above arguments, we can analogously show that,
 for every  vertex $u \in \widetilde W$,
if  $\widetilde \alpha_Y(u) = v_2(p_Y)$ for some standard $z_4$-path  $p_Y$ of $Y (f_{\ell+1}, z_4)$, then  $\widetilde \alpha_X(u) \in p_X$, where $p_X$ is a  standard  $z_4$-path in $X (f_{\ell+1}, z_4)$.

Now we can use Lemma~\ref{wd42} to conclude that, for every  vertex $u \in \widetilde W$,
if  $\widetilde \alpha_X(u) = v_2(p_X)$, where $p_X$  is a standard  $z_4$-path  of $X (f_{\ell+1}, z_4)$, or if  $\widetilde \alpha_Y(u) = v_2(p_Y)$, where $p_Y$ is a standard  $z_4$-path  of $Y (f_{\ell+1}, z_4)$, then,  in either case,
$\widetilde \alpha_X(u) = v_2(p_X)$ and $\widetilde \alpha_Y(u) = v_2(q_Y)$, where both $p_X$, $p_Y$ are standard  $z_4$-paths. In addition, there exists a path $p_W$ in $\widetilde W$ such that
$\widetilde \alpha_X(p_W) = p_X$ and $\widetilde \alpha_Y(p_W) = p_Y$. Consequently,  if $p$ is a path in
$\widetilde W$ such that one of $\widetilde \alpha_X(p)$, $\widetilde \alpha_Y(p) $ is a standard $z_4$-path,
then both $\widetilde \alpha_X(p)$, $\widetilde \alpha_Y(p) $ must be standard  $z_4$-paths. Now the desired equality  $\widetilde W = W (f_{\ell+1}, z_4) $ becomes apparent.

Since $\barr (X) = \barr (X(f_{\ell+1}, z_4)) = \barr (\mu(X^\tau) ) $
and $\barr (Y) = \barr (Y(f_{\ell+1}, z_4))  = \barr (\mu(Y^\tau) ) $,  this $(f_{\ell+1}, z_4(w))$-transformation  over the graphs  $W^\tau$, $X^\tau$, $Y^\tau, Z$  is conservative, as required.
It remains to note that  $A_{f_{\ell+1}}$-graphs
$$
X (f_{\ell+1}, z_4) = \mu(X^\tau) ,  \quad Y (f_{\ell+1}, z_4)= \mu(Y^\tau)
$$
are $A_{f_{\ell+1}}$-incomplete and Lemma~\ref{cntr} is proven.
\end{proof}

{\em Proof of Theorem~\ref{Th1}}. Since neither of subgroups $H, K$ has finite index in
$ L = \langle H, K, S(H,K)   \rangle$, we conclude that $\rr( L ) \ge 2$. There is nothing to prove if
$\rr( L ) = 2$. Hence, we may assume that $\rr( L ) = n  >2$.
It follows from Lemma~\ref{cntr} and the definitions  that there is an
epimorphism $\e_1 : L \to F_{n-1}$, where $F_{n-1}$ is a free group of rank $n-1$,  with the following properties.  The restriction of $\e_1$ on $H$ and on $K$ is injective, a set
$S(\e_1(H), \e_1(K) )$ for subgroups $\e_1(H), \e_1(K)$ of $F_{n-1}$ can be taken to be $\e_1(S(H,K))$,
 for every $s \in S(H,K)$, the restriction
 $$
 \e_{1,s} : H \cap s K s^{-1} \to  \e(H)  \cap \e(s) \e(K) \e(s)^{-1}
 $$
of $\e_1$ on   $H \cap s K s^{-1}$  to   $\e(H)  \cap \e(s) \e(K) \e(s)^{-1}$ is surjective, and  neither  of  $\e_1(H), \e_1(K)$   has finite index in the group
$$
F_{n-1} = \e_1(L) =  \langle \e_1(H), \e_1(K), \e_1(S(H,K)) \rangle .
$$
Iterating this argument, we obtain a desired epimorphism $\e : L \to F_{2}$. \qed

\smallskip

{\em Proof of Corollary~\ref{Cor}}. If one of subgroups $H$, $K$ has infinite index in their join
$\langle H, K \rangle$ then one of $H$, $K$ has also infinite index in the generalized join  $ L = \langle H, K, S(H,K)   \rangle$ and Theorem~\ref{Th1} applies. By Theorem~\ref{Th1}, there is an epimorphism $\e : L \to F_{2}$ that has the required properties of $\dl$. \qed

\end{document}